\documentclass[11pt]{article}
\usepackage{amssymb,amsfonts,amsmath,latexsym,epsf,tikz,url}
\usepackage[usenames,dvipsnames]{pstricks}
\usepackage[linesnumbered,ruled,vlined]{algorithm2e}
\usepackage{pstricks-add}
\usepackage{graphicx} 
\usepackage{epsfig}
\usepackage{pst-grad} 
\usepackage{pst-plot} 
\usepackage[space]{grffile} 
\usepackage{etoolbox} 
\makeatletter 
\patchcmd\Gread@eps{\@inputcheck#1 }{\@inputcheck"#1"\relax}{}{}
\makeatother

\newtheorem{theorem}{Theorem}[section]
\newtheorem{proposition}[theorem]{Proposition}

\newtheorem{corollary}[theorem]{Corollary}
\newtheorem{lemma}[theorem]{Lemma}

\newtheorem{definition}[theorem]{Definition}

\newcommand{\qed}{\hfill $\square$\medskip}

\begin{document}

\title{Counting the Number of Domatic Partition of a Graph}

\author{ 
	Saeid Alikhani$^{1,}$\footnote{Corresponding author}\and Davood Bakhshesh$^2$ \and Nima Ghanbari$^1$ 
}

\date{\today}

\maketitle

\begin{center}
	
	$^1$Department of Mathematical Sciences, Yazd University, 89195-741, Yazd, Iran\\	
	
	$^2$Department of Computer Science, University of Bojnord, Bojnord, Iran
	
	\medskip
	\medskip
	{\tt  ~~
		alikhani@yazd.ac.ir, d.bakhshesh@ub.ac.ir, n.ghanbari.math@gmail.com
	}
	
\end{center}

\begin{abstract}
A subset of vertices $S$ of a graph $G$ is a dominating set if every vertex in $V \setminus S$
has at least one neighbor in $S$. A domatic partition is a partition of the vertices of a graph $G$ into disjoint dominating sets. The domatic number $d(G)$ is the maximum size of a domatic partition. Suppose that  $dp(G,i)$  is the number of distinct domatic partition of $G$ with cardinality $i$. 
In this paper, we consider  the generating function of $dp(G,i)$, i.e., $DP(G,x)=\sum_{i=1}^{d(G)}dp(G,i)x^i$ which we call it  the domatic partition  polynomial.  We explore the domatic polynomial for trees, providing a quadratic time algorithm for its computation based on weak 2-coloring numbers. Our results include specific findings for paths and certain graph products, demonstrating practical applications of our theoretical framework.
\end{abstract}

  
\medskip
\noindent{\bf AMS Subj.\ Class.}:  05C69.

\section{Introduction}

Let $G = (V,E)$ be
a simple graph of order $n$. The open neighborhood (closed neighborhood) of a vertex
$v \in V$ is the set $N(v)$ = $\{u | uv \in E\}$, (the set
$N[v]$ = $N(v) \cup \{v\}$). The number of vertices in  $|N(v)|$
is the degree of $v$, denoted by $ deg(v)$.  A set $S\subseteq V$ is a dominating set of a graph $G$, if every vertex in $V \setminus S$
has at least one neighbor in $S$, in other words $N[S]=V$. The cardinality of a minimum  dominating set in $G$ is called the domination number of $G$ and is denoted by $\gamma(G)$. 
The various different domination concepts are
well-studied now, however new concepts are introduced frequently and the interest is growing
rapidly. We recommend two fundamental books \cite{domination,2} and some surveys \cite{6,5} about
domination in general.

A domatic partition is a partition
of the vertex set into dominating sets, in other words, a partition $\pi$ = $\{V_1, V_2, . . . , V_k \}$ of $V(G)$  
such that every set $V_{i}$ is a dominating set in $G$. 
Cockayne and Hedetniemi \cite{Cockayne} introduced the domatic number of a graph $d(G)$ as the maximum order $k$ of a vertex partition. For more details on the domatic number refer to e.g., \cite{11,12,13}.

Motivated by enumerating of  the number of dominating sets of a graph and domination polynomial (see e.g. \cite{euro}), the enumeration of  the domatic partition for certain
graphs is a natural subject.  In other words,  we explore domatic partition  from
the point of view of the counting polynomial defined in the following standard
way.

\begin{definition}\label{def}
			Let ${\cal DP}(G,i)$ be the family of
		domatic partition of a graph $G$ with cardinality $i$, and let
		$dp(G,i)=|{\cal DP}(G,i)|$. The domatic polynomial $DP(G,x)$ of $G$ is defined as
		$$DP(G,x)=\sum_{i=1}^{d(G)}dp(G,i)x^i,$$
		where $d(G)$ is the domatic number of $G$.
\end{definition}

In Section 2, we explore various properties of the domatic polynomial. Moving on to Section 3, we delve into the investigation of the domatic polynomial for trees. Following that, in Section 4, we introduce a quadratic time algorithm designed to compute the domatic polynomial of trees, accompanied by an analysis of its time complexity. Lastly, we wrap up our paper with a conclusion in the final section.

\section{Introduction to domatic polynomial}

In this section, we obtain some properties of domatic polynomial of a graph.
We need the following result:

	\begin{theorem}{\rm\cite{Cockayne}}\label{thm:domatic-min-deg}
For any graph $G$, $d(G)\leq \delta +1$, where $\delta$ is the minimum degree of $G$. 
	\end{theorem}

Also, we need the following easy lemma. It is well known that if there is no isolated vertex in the graph, i.e. $\delta\geq 1$, then the domatic number is at least 2. For convenience, we present a proof.

	\begin{lemma}\label{lem:lowerbound}
If  $G$ is a connected graph, then  $d(G)\geq 2$.
	\end{lemma}
	
	\begin{proof}
Suppose that $G$ is a graph of order $n$, and $D$ is a dominating set of that with minimum size. Since $G$ is connected, so by Ore's Theorem  $\gamma(G)\leq \frac{n}{2}$. Let $D$ be a dominating set of $G$ with minimum size. By the definition of dominating sets, then  $\overline{D}$ is a dominating set of $G$ too. So 
$$P=\{D,\overline{D}\}$$
is  a domatic partition of $G$, and therefore we have the result. 
\qed
	\end{proof}

As an immediate result of definition of domatic number, $dp(G,i)=0$ for $i>d(G)$. So the following definition is equivalent to the definition of domatic polynomial of a graph by using Theorem \ref{thm:domatic-min-deg}:

	\begin{definition}\label{def:equ}
{\rm	
Let ${\cal DP}(G,i)$ be the family of
domatic partition of a graph $G$ with cardinality $i$, and let
$dp(G,i)=|{\cal DP}(G,i)|$. The domatic polynomial $DP(G,x)$ of $G$ is defined as
$$DP(G,x)=\sum_{i=1}^{\delta +1}dp(G,i)x^i,$$
where $\delta$ is the minimum degree of $G$.
}
	\end{definition}

As an immediate result of Definition \ref{def:equ}, we have the following result. 

	\begin{proposition}\label{prop:isolate}
If $G$ has isolate vertices, then	
$$DP(G,x)=x.$$
	\end{proposition}

In \cite{Garey}, it was shown that finding domatic number of a graph is NP-complete. Consequently, we have the following result.

	\begin{theorem}
Computation of the domatic polynomial of a graph is NP-complete.
	\end{theorem}

A weak $k$-coloring of a graph $G=(V,E)$ assigns a color $c(v)\in \{1,2,\ldots,k\}$ to each vertex $v\in V$, such that each non-isolated vertex is adjacent to at least one vertex with different color. So a weak $2$-coloring of a graph is equivalent to finding a domatic partition of a graph of size $2$. In the following, let ${\cal W}(G,2)$ be the family of
weak $2$-coloring of a graph $G$, and let
$w_2(G)=|{\cal W}(G,2)|$. So $dp(G,2)=w_2(G)$.  

We conclude this section, with the following results which are  immediately obtained by the Definition \ref{def}. 
\begin{proposition}\label{prop:properties}
	Let $G$ be a graph, and $d(G)=r$. The following holds:	
	\begin{itemize}
		\item[(i)]
		$dp(G,1)=1.$
		\item[(ii)]
		$dp(G,2)$ is the number of weak 2-coloring of $G$.
		\item[(iii)]
		$DP(G,1)$ is the number of all domatic partition of $G$.
		\item[(iv)]
		$\frac{\mathrm{d^{r}}}{\mathrm{d}x^{\mathrm{r}}}DP(G,x)=r!dp(G,r).$
		\item[(v)]
		$\frac{\mathrm{d^{s}}}{\mathrm{d}x^{\mathrm{s}}}DP(G,x)=0,$ for $s>r$.
		\item[(vi)]
		Zero is a root of $DP(G,x)$, with multiplicity one.
	\end{itemize}
\end{proposition}

\section{Domatic Polynomial of Trees}
In this section, we aim to compute the domatic polynomial for a given  tree. By Theorem~\ref{thm:domatic-min-deg} and Lemma \ref{lem:lowerbound}, if $T$ is a tree, then $d(T)= 2$. So we have the following result:

	\begin{proposition}\label{prop:tree}
		Let $T$ be a tree. Then 	$d(T)= 2$, and 
		$DP(T,x)=x+w_2(T)x^2$.
	\end{proposition}
	Let $T$ be a tree and $u$ be a support vertex. Let $ones(u)$ be  the collection of vertices that have a degree of 1 and share an edge with $u$. Let $T_u^{-1}$ and $T_u^{-2}$ be two trees obtained from $T$ by removing the vertices $ones(u)$ and $ones(u)\cup \{u\}$, respectively.
	
	In a tree, we call a vertex $v$ as a  {\em quasi-star vertex}  of  tree  if it  satisfies the following conditions:
	\begin{itemize}
		\item $v$ is adjacent to exactly one non-leaf vertex (an internal vertex).
		\item $v$ is adjacent to at least one leaf (vertex with degree 1).
	\end{itemize} 
We call a tree  containing a quasi-star vertex by {\em star-neighbor} tree.  
 Now, we prove the following result. 
	 \begin{theorem}
	 	\label{thmnt}
	 	For any star-neighbor tree $T$ of order $n\geq 4$, if $y$ is its quasi-star vertex, then 
	 	$$w_2(T)=w_2(T_y^{-1})+w_2(T_y^{-2}),$$ 
	 	where $w_2(K_{1,r})=1$, for any $r\ge 1$.
	 	\end{theorem}
	 \begin{proof}
	 	It is clear that  for any $r\ge 1$, $w_2(K_{1,r})=1$. Suppose that $\cal{C}$ is a weak 2-coloring for tree $T$. We will show that with this coloring, if we cannot construct a weak 2-coloring for tree $T_y^{-1}$, we can certainly construct a weak 2-coloring for tree $T_y^{-2}$ using it. Also, suppose that the non-leaf vertex adjacent to vertex $y$ is vertex $w$. It is easy to see that if the colors of vertices $y$ and $w$ are different, then coloring $\cal{C}$ is indeed a coloring for tree $T_y^{-1}$. Now, assume that in coloring $\cal{C}$, the colors of the two vertices $y$ and $w$ are the same. Therefore, it is clear that since vertex $y$ is a leaf in tree $T_y^{-1}$, coloring $\cal{C}$ is not a coloring for $T_y^{-1}$, but since vertex $w$ must, by the definition of weak 2-coloring, be adjacent to a vertex with a different color, one of the vertices adjacent to $w$ must have a different color from $w$. Therefore, we conclude that coloring $\cal{C}$ is a coloring for tree $T_y^{-2}$.
	 	
	 	Based on the above discussion, we can derive the  relationship 
	 	$w_2(T)\le w_2(T_y^{-1})+w_2(T_y^{-2})$. 
	 	
	 	Now, since any coloring for tree $T_y^{-1}$ or $T_y^{-2}$ can easily be transformed into a coloring for tree $T$, the above relationship is an exact equality. Therefore, the theorem is proved.\qed
	 	\end{proof}
	 	
	 	Now,  consider  a path $P_n$. It is clear that $P_n$  is a star-neighbor tree. Let $T:=P_n$. It is clear that $T_y^{-1}$ or $T_y^{-2}$ are $P_{n-1}$ and $P_{n-2}$, respectively. Now, using Theorem \ref{thmnt}, we have the following result. 
	 		 \begin{corollary}
	 		\label{co1}
	 		For any path $P_n$ of order $n\geq 4$, 
	 		$$w_2(P_n)=w_2(P_{n-1})+w_2(P_{n-2}),$$ 
	 		where $w_2(P_3)=w_2(P_2)=1$.
	 	\end{corollary}
	 	As an immediate result of Corollary \ref{co1}, we have:
	 	\begin{corollary}
	 		$\lim_{n \to \infty}\frac{w_2(P_{n+1})}{w_2(P_n)}=\phi,$
	 		where $\phi$ is the golden ratio.
	 	\end{corollary}	 	
	 	\begin{proof}
	 		Since $w_2(P_n)$ follows the Fibonacci sequence, we have the result.\qed
	 	\end{proof}

	 	Let $w_1,w_2, \ldots , w_k$ be some vertices of a graph $G$.  Let $B(G;w_1,...,w_k)$ be the bouquet of graph $G$  with respect to the vertices $\{w_i\}_{i=1}^k$ and obtained by identifying the vertex $w_i$ of the graph with vertex $w$. 
	 
	 \bigskip
	 
	 Consider a tree denoted by $T$, and select a support vertex within $T$, labeled $y$. Let $N(y) \setminus ones(y)=\{w_1, w_2, \ldots, w_k\}$. The tree obtained by bouqueting the vertices $\{w_1, w_2, \ldots, w_k\}$ is denoted as $T'$ (see Figure~\ref{fs1}).	 
	\begin{figure}[ht]
	\begin{center}
		\includegraphics[width=0.85\linewidth]{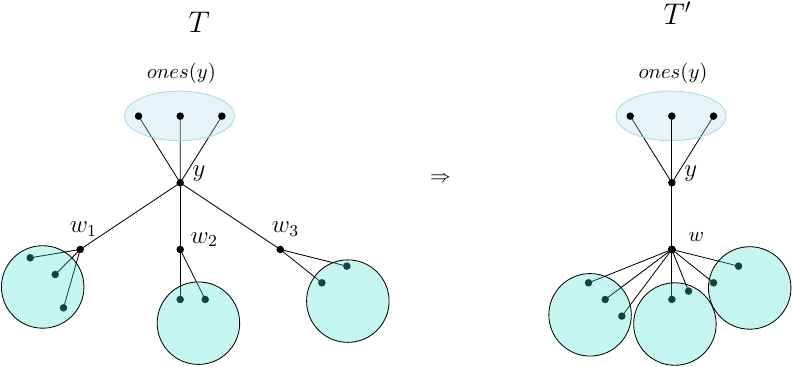}
		\caption{A tree $T$  and the tree  $T'$.}
		\label{fs1}
	\end{center}
\end{figure}
Now, we prove the following result.
\begin{theorem}
	\label{thmw}
For any tree $T$ of order $n>1$,   if $y$ is one of  its support vertices, then
$$w_2(T)=w_2(T_y^{-1})+w_2({T'}_y^{-2}).$$
\end{theorem}
\begin{proof}
	 Suppose that $\cal{C}$ is a weak 2-coloring for tree $T$. We will show that with this coloring, if we cannot construct a weak 2-coloring for tree $T_y^{-1}$, we can certainly construct a weak 2-coloring for tree ${T'}_y^{-2}$ using it. Also, let $N(y) \setminus ones(y)=\{w_1, w_2, \ldots, w_k\}$. It is easy to see that if the colors of vertices $y$ and one of the vertices $\{w_1, w_2, \ldots, w_k\}$ are different, then coloring $\cal{C}$ is indeed a coloring for tree $T_y^{-1}$. Now, assume that in coloring $\cal{C}$, the colors of the vertices $y$ and all vertices $\{w_1, w_2, \ldots, w_k\}$ are the same. Now, by bouqueting the vertices  $\{w_1, w_2, \ldots, w_k\}$ into  a vertex $w$, we obtain the tree $T'$. We color the vertex $w$ by the color assigned to the vertices $\{w_1, w_2, \ldots, w_k\}$.    Since  $w$ have the same color with $y$,  coloring $\cal{C}$ is not a coloring for ${T'}_y^{-1}$, but since the vertex $w$ in $T'$ must, by the definition of weak 2-coloring, be adjacent to a vertex with a different color, one of the vertices adjacent to $w$ must have a different color from $w$. Therefore, we conclude that coloring $\cal{C}$ is a coloring for tree ${T'}_y^{-2}$.
	
	Based on the above discussion, we can derive the  relationship 
$w_2(T)=w_2(T_y^{-1})+w_2({T'}_y^{-2})$. 
	
	Now, since any coloring for tree $T_y^{-1}$ or ${T'}_y^{-2}$ can easily be transformed into a coloring for tree $T$, the above relationship is an exact equality. Therefore, the theorem is proved.
	\qed
\end{proof}

As an another conclusion of Theorem \ref{thmw}, we have the following result. 
\begin{proposition}\label{prop:pnok1}
	If  $P_n$ is a  path  of order $n$, then $$DP(P_n\circ K_1,x)=x+2^{n-1}x^2.$$	
\end{proposition}
\begin{proof}
	It is not hard to see that by Theorem \ref{thmw}, $w_2(P_n\circ K_1)=2w_2(P_{n-1}\circ K_1)$. Since $w_2(P_{2}\circ K_1)=2$, we easily conclude that $w_2(P_n\circ K_1)=2^{n-1}$. Hence, $DP(P_n\circ K_1,x)=x+2^{n-1}x^2.$
\end{proof}

Finally, we extend the previous result as follows. 
\begin{theorem}
	Let $G$ be a graph of order $n$, and let $r>0$ be an integer. Then 
	$$DP(G\circ \overline{K}_r,x)=x+2^{n-1}x^2.$$
\end{theorem}
\begin{proof}
	Assume we have two colors $c_1$ and $c_2$ for coloring of $G$. As we know, each vertex of the graph $G$ is adjacent to at least $r$ vertices of degree one. Therefore, for a weak 2-coloring  of $G$, any color assigned to vertex $x$ must be different from the colors assigned to all its adjacent vertices of degree one. Hence, the method of coloring is completely determined. Now, since each vertex in $G$ has two coloring options, the total number of ways to color (weak 2-coloring) the graph is $2^n$. However, by swapping color $c_1$ with color $c_2$, we obtain another coloring, but in terms of domatic partitioning, it is no different from the coloring before the swap. Therefore, the total number of distinct weak 2-colorings of the graph is $2^{n-1}$. Hence, $w_2(G\circ \overline{K}_)=2^{n-1}$ that completes the proof.\qed
\end{proof}
\section{Algorithm and analysis}
Based on Theorem \ref{thmw}, we provide the following algorithm that computes the $w_2(T)$ for a tree $T$.

	\begin{algorithm}[H]
		\SetAlgoLined
		\KwIn{Tree \( T \)}
		\KwOut{\( w_2(T) \)}
		
		\SetKwFunction{ComputeWTwo}{ComputeW2}
		\SetKwProg{Fn}{Function}{:}{}
		
		\Fn{\ComputeWTwo{$T$}}{
			\If{\( T \) has only one vertex}{
				\Return 0\;
			}
			Identify a support vertex \( y \) in \( T \)\;
			Compute \( ones(y) \) (vertices of degree 1 adjacent to \( y \))\;
			Compute \( T_y^{-1} \) by removing vertices in \( \text{ones}(y) \)\;
			Compute \( {T'} \) by bouqueting the vertices  \( \{w_1, w_2, \ldots, w_k\} \) in \( T \)\;
			Compute \( {T'}_y^{-2} \) by further removing vertex \( y \) from \( {T'}_y^{-1} \)\;
		}
		
		\BlankLine
		 \Return \ComputeWTwo{$T_y^{-1}$} + \ComputeWTwo{${T'}_y^{-2}$}\;
		\caption{Algorithm to compute \( w_2(T) \) in a tree \( T \)}
	\end{algorithm}

Now, we analyze  the time complexity of the algorithm.
	
	To analyze the time complexity of the algorithm for computing $w_2(T)$ for a tree $T$ of order $n$, we consider the following steps:
	\begin{itemize}
		\item \textbf{Step 1: Identifying a support vertex \( y \)}:
		\begin{itemize}
			\item Traversing the tree and checking the degrees of vertices takes \( O(n) \) time.
		\end{itemize}
		
		\item \textbf{Step 2: Computing \({ones}(y) \)}:
		\begin{itemize}
			\item Identifying the vertices of degree 1 adjacent to \( y \) can be done in \( O(d_y) \) time, where \( d_y \) is the degree of \( y \).
			\item Since \( d_y \) can be at most \( n-1 \), this step is bounded by \( O(n) \).
		\end{itemize}
		
		\item \textbf{Step 3: Computing \( T_y^{-1} \)}:
		\begin{itemize}
			\item Removing the vertices in \( {ones}(y) \) involves visiting each of these vertices and removing them, which can be done in \( O(d_y) \) time.
			\item This step is bounded by \( O(n) \) since \( d_y \leq n-1 \).
		\end{itemize}
		
		\item \textbf{Step 4: Computing \( {T'} \)}:
		\begin{itemize}
			\item Contracting vertices \( \{w_1, w_2, \ldots, w_k\} \) into a single vertex involves merging these vertices and updating the edges.
			\item This can be done in \( O(n) \) time.
		\end{itemize}
		
		\item \textbf{Step 5: Computing \( {T'}_y^{-2} \)}:
		\begin{itemize}
			\item Removing vertex \( y \) from \( {T'} \) can be done in \( O(1) \) time.
		\end{itemize}
		
		\item \textbf{Step 6: Recursive computation of \( w_2 \)}:
		\begin{itemize}
			\item The function \( w_2 \) is computed recursively on \( T_y^{-1} \) and \( {T'}_y^{-2} \).
			\item Let \( n_1 \) and \( n_2 \) be the sizes of the smaller subtrees formed after removing vertices from \( T \). The combined size is \( n_1 + n_2 \leq n \).
		\end{itemize}
	\end{itemize}
	
	The recurrence relation for the time complexity \( T(n) \) is:
	\[
	T(n) = T(n_1) + T(n_2) + O(n).
	\]
	
	Using the Master Theorem for divide-and-conquer recurrences, we analyze this recurrence:
	
	\begin{itemize}
		\item \textbf{Case 1: Balanced subproblems}
		\[
		T(n) = 2T\left(\frac{n}{2}\right) + O(n).
		\]
		According to the Master Theorem, this recurrence has a solution of \( T(n) = O(n \log n) \).
		
		\item \textbf{Case 2: Highly unbalanced subproblems}
		\[
		T(n) = T(n-1) + O(n).
		\]
		This recurrence results in a time complexity of \( T(n) = O(n^2) \).
	\end{itemize}

	Now, we finalize this section. 
	\begin{theorem}
		The worst-case  time complexity of the algorithm for computing $w_2(T)$ for a tree $T$ of order $n$ is $O(n^2)$. 
	\end{theorem}

\section{Conclusion}

In this paper, we have introduced and analyzed the domatic partition polynomial of a graph, $DP(G,x)$, which enumerates the domatic partitions of $G$ by their cardinality. We have established several properties of this polynomial, including its relationship with the minimum degree of the graph and its computational complexity. Specifically, we demonstrated that computing the domatic polynomial is NP-complete.

We also focused on trees and provided a detailed examination of their domatic polynomials. We derived a quadratic time algorithm for computing the domatic polynomial of a tree by leveraging the weak 2-coloring number, $w_2(T)$. This algorithm capitalizes on the hierarchical structure of trees, recursively breaking down the problem into smaller subtrees. 

Furthermore, we provided specific results for paths and certain graph products, showcasing the practical applications of our theoretical findings. For paths $P_n$, we determined that $DP(P_n, x) = x + 2^{\lfloor\frac{n}{2}\rfloor-1}x^2$. Additionally, we extended our results to graphs of the form $G \circ \overline{K}_r$, demonstrating that $DP(G \circ \overline{K}_r, x) = x + 2^{n-1}x^2$ for a graph $G$ of order $n$.

Overall, our work provides a comprehensive framework for understanding and computing the domatic partition polynomial, opening new avenues for future research in graph theory and combinatorial optimization.

\end{document}